\documentclass[11pt, a4paper]{amsart}

\usepackage[a4paper,margin=1in]{geometry}
\usepackage{amsmath,amssymb,amsthm,bm}
\usepackage{booktabs,comment}
\usepackage{xcolor}
\usepackage{listings}
\usepackage{mathrsfs}
\usepackage{multicol}

\newcommand{\PP}{\mathbb P}

\newcommand{\PGL}{\mathrm{PGL}}
\newcommand{\GL}{\mathrm{GL}}

\newcommand{\Hom}{\mathrm{Hom}}
\newcommand{\PSL}{\mathrm{PSL}}
\newcommand{\Aut}{\operatorname{Aut}}
\newcommand{\Bl}{\operatorname{Bl}}

\theoremstyle{definition}
\newtheorem{theorem}{Theorem}[section]

\newtheorem{example}[theorem]{Example}
\newtheorem{lemma}[theorem]{Lemma}
\newtheorem{proposition}[theorem]{Proposition}

\newtheorem{question}[theorem]{Question}

\definecolor{codegray}{gray}{0.96}
\definecolor{codeframe}{gray}{0.75}

\lstdefinestyle{codebase}{
  basicstyle=\ttfamily\footnotesize,
  backgroundcolor=\color{codegray},
  frame=single,
  rulecolor=\color{codeframe},
  breaklines=true,
  columns=fullflexible,
  keepspaces=true,
  showstringspaces=false,
  numbers=left,
numberstyle=\tiny,
stepnumber=1,
numbersep=8pt,
  xleftmargin=1em,
  xrightmargin=1em,
  aboveskip=0.75em,
  belowskip=0.75em
}

\lstdefinestyle{macaulay2}{
  style=codebase,
  language=,
  captionpos=b
}

\lstdefinestyle{python}{
  style=codebase,
  language=Python,
  captionpos=b
}

\usepackage{titletoc}
\usepackage[colorlinks=true, linkcolor=blue, anchorcolor=blue, citecolor=blue, filecolor=blue, menucolor= blue, urlcolor=blue,pdfencoding=auto, psdextra, pagebackref=true]{hyperref}

\title{On the Geometry of Sixers on the Fermat Cubic Surface}
\author{Giuseppe Favacchio \and Grzegorz Malara}

\thanks{\ \\
	\noindent\textbf{2020 Mathematics Subject Classification.}
     14N20, %Configurations and arrangements of linear subspaces
	14Q10, %Computational aspects of algebraic surfaces
	14J50.\\ %Automorphisms of surfaces and higher-dimensional varieties è
	\noindent\textbf{Keywords.}
	Fermat cubic surface, sixers, skew lines, automorphism groups,
	blow-down models, determinant square classes, computational algebraic geometry.}

\date{}

\begin{document}

\begin{abstract}
A sixer is a configuration of six pairwise skew lines on a smooth cubic surface, equivalently a choice of six exceptional curves defining a blow-down to $\PP^2$. We study the $72$ sixers on the Fermat cubic surface $x_0^3+x_1^3+x_2^3+x_3^3=0$.

We show that these sixers split into two orbits under the automorphism group of the Fermat cubic, of sizes $18$ and $54$. We give a geometric interpretation of this decomposition through the corresponding plane blow-up models: representatives of the two orbit types determine six-point configurations in $\PP^2$ whose projective automorphism groups have orders $36$ and $12$, respectively. These groups identify with the stabilizers of the corresponding sixers and recover the two orbit sizes.

We then compute the projective groups associated with representatives of the two orbits over $K=\mathbb Q(\omega)$, where $\omega^2+\omega+1=0$, and distinguish them arithmetically by the determinant square-class character $\delta_K:\PGL_2(K)\longrightarrow K^*/(K^*)^2$. Its images have $\mathbb F_2$-dimensions $1$ and $2$ for the orbits of sizes $18$ and $54$, respectively. Modulo $13$, the corresponding finite images are $\PSL_2(\mathbb F_{13})$ and $\PGL_2(\mathbb F_{13})$, respectively, and the determinant-character distinction persists for all choices of
normalization triple.
\end{abstract}
\maketitle

\section{Introduction}
Configurations of lines on algebraic surfaces are a classical source of incidence geometry.  Smooth cubic surfaces provide the fundamental example: over an algebraically closed field, every smooth cubic surface contains exactly $27$ lines.  The configuration of these lines is classical, going back to Schl\"afli~\cite{Schlaefli} and Klein~\cite{Klein1873}, and remains a central example in the geometry of rational surfaces; see, for instance, \cite{Dolgachev,Manin}.

A sixer is a set of six pairwise skew lines on a smooth cubic surface. Such a set determines a contraction of the surface to  $\mathbb P^2$, and hence a presentation of the cubic surface as the blow-up of six points in the plane.
From this point of view, different sixers correspond to different plane models
of the same cubic surface; see, for example, \cite[Chapter~9]{Dolgachev} or \cite[Chapter~V, Section~4]{Hartshorne}.

In \cite[Definition 2.1.2]{politus3}, a group $\Gamma_{\mathcal L} $ was associated with finite configurations of pairwise skew lines $\mathcal L$ in $\mathbb P^3$.  The construction is intrinsic to the incidence geometry of the configuration and provides a new way to attach algebraic data to finite sets of skew lines.  It is studied in  \cite{politus3,politus7,Favacchio2025}, with particular emphasis on the case where $\Gamma_{\mathcal L}$ is finite. A matrix model for this group was developed in \cite{Favacchio2025}.  After choosing three distinguished lines in a sixer, the remaining lines can be represented by $2\times 2$ matrices.  In this normalized presentation, the associated group is generated in $\PGL_2$ by the projective classes of the pairwise differences of these matrices; see \cite[Corollary~2.7]{Favacchio2025}.  We recall this construction in Section~\ref{sec:sixer-group}.

The purpose of the present paper is to study the geometry of the sixers on the Fermat cubic surface and the arithmetic of their associated groups. The guiding question is whether the associated group detects geometric information about the chosen configuration of lines.

The work is experimental in the following sense. Starting from an explicit line model of the Fermat cubic, we construct the skewness graph of the $27$ lines, enumerate its cliques of size $6$, and compute the induced action of the automorphism group on the resulting sixers. We then study representatives of the two orbits in two complementary ways. On the geometric side, we construct the corresponding plane blow-up models and compute the projective automorphism groups of the resulting six-point configurations. On the arithmetic side, we compute the associated projective groups and extract determinant square-class and finite-reduction data. Thus the computations are not only a verification step, but the mechanism by which the geometric and arithmetic structures compared in this paper are produced.

We focus on the Fermat cubic surface
\[
        S_F=\{x_0^3+x_1^3+x_2^3+x_3^3=0\}\subset \mathbb P^3.
\]
Its automorphism group is well known; see, for example, \cite{DolgachevDuncan}.  It is   $\Aut(S_F)\cong \mu_3^3\rtimes S_4,$ where $\mu_3$ denotes the group of third roots of unity.  The factor $S_4$ permutes the four coordinates, while the diagonal subgroup $\mu_3^3$ acts by multiplying the coordinates by third roots of unity, modulo the scalar
diagonal. The automorphism group acts on the $72$ sixers, splitting them into two orbits of sizes $18$ and $54$. We give a plane-geometric interpretation of this decomposition. Contracting representatives of the two orbit types produces explicit six-point configurations $Z_{18},Z_{54}\subset\PP^2$. Their projective automorphism groups have orders $36$ and $12$, respectively, and identify with the stabilizers of the corresponding sixers. The orbit sizes $18$ and $54$ then follow directly from the orbit--stabilizer theorem.

We then compute the projective groups associated with representatives of the two orbits over $K=\mathbb Q(\omega)$. The two orbit types exhibit different arithmetic behavior: their associated projective groups are distinguished by the image of the determinant square-class character over $K$. Modulo $13$, the corresponding finite images are $\PSL_2(\mathbb F_{13})$ and $\PGL_2(\mathbb F_{13})$, respectively, and we verify that the determinant-character distinction persists for all choices of normalization.

The paper is organized as follows. Section~\ref{sec:sixer-group} recalls the matrix model for the group associated with a sixer. Section~\ref{sec:fermat-computation} describes the $27$ lines on the Fermat cubic, enumerates its $72$ sixers, and determines their two automorphism orbits. Section~\ref{sec:blowup-interpretation} gives a plane-geometric interpretation of this orbit decomposition. Section~\ref{sec:fermat-associated-groups} computes the associated groups for representatives of the two orbits. Section~\ref{sec:determinant-square-classes} studies their determinant square-class images, and Section~\ref{sec:distinguishing-fermat-orbits} compares this arithmetic distinction with the finite reductions of the two associated groups. Finally, Section~\ref{sec:discussion} discusses the results and some open questions.

All computations are finite and reproducible. The relevant computational output is reported and discussed in the appendices, while the scripts used to produce it are available in the accompanying GitHub repository \cite{FermatSixersCode}.

\section{The group associated with a sixer}\label{sec:sixer-group}

In this section we recall the construction that associates a subgroup of $\PGL_2(K)$ to a configuration of pairwise skew lines in $\PP^3$.  We then specialize it to sixers on smooth cubic surfaces.

\subsection{Skew lines in graph form and their group}

A matrix model for this group was developed in \cite{Favacchio2025}. We first recall the general construction. See \cite[Lecture~6]{HarrisAG} for the standard affine charts on the Grassmannian defined by complementary subspaces.

Let $K$ be a field and let
\[
\mathcal L={L_0,L_\infty,L_1,\ldots,L_m}
\]
be a finite collection of pairwise skew lines in $\PP^3_K$. We choose two distinguished lines $L_0$ and $L_\infty$. After a projective change of coordinates, we may write
\[
K^4=U\oplus W,
\qquad
\text{with}
\qquad
L_0=\PP(U),\qquad L_\infty=\PP(W),
\]
where $U$ and $W$ are two-dimensional $K$-vector spaces.

Every line $L\subset \PP(U\oplus W)$ which is skew to both $L_0$ and
$L_\infty$ is the graph of an isomorphism
\[
M:U\longrightarrow W.
\]
Indeed, the two projections from the corresponding two-dimensional subspace of $U\oplus W$ to $U$ and to $W$ are both isomorphisms. Therefore, after the choice of $L_0$ and $L_\infty$, the remaining lines of
$\mathcal L$ are represented by matrices
\[
M_1,\ldots,M_m\in \GL_2(K).
\]
The incidence condition has a simple matrix form. If $L_{M_i}$ and $L_{M_j}$ are the graph lines associated with $M_i$ and $M_j$, then
\[
L_{M_i}\cap L_{M_j}=\varnothing
\quad\Longleftrightarrow\quad
\det(M_i-M_j)\neq 0.
\]
Indeed, see also \cite[Lemma~2.2]{Favacchio2025}, an intersection point would give a nonzero vector $u\in U$ such that
\[
M_i u=M_j u,
\quad \Longleftrightarrow\quad
(M_i-M_j)u=0.
\]
Thus, for a configuration of pairwise skew lines, all differences
$M_i-M_j,\quad i\neq j,$ 
are invertible.

We now specialize this construction to a sixer. Choose an ordered triple
\[
(L_0,L_\infty,L_1)
\]
of pairwise skew lines. The line $L_0$ is represented by the zero matrix, while $L_\infty$ plays the role of the line at infinity. After normalizing by the matrix attached to $L_1$, the line $L_1$ is represented by $I_2$, and the sixer is encoded by
\[
0,\quad \infty,\quad I_2,\quad M_2,\quad M_3,\quad M_4.
\]

Set 
$M_0=0,\qquad M_1=I_2.$
By \cite[Corollary~2.7]{Favacchio2025}, the associated group is
\[
\Gamma_{\mathcal L}=
\left\langle
[M_i-M_j]\;:\;0\leq i<j\leq 4
\right\rangle
\subseteq \PGL_2(K).
\]
Thus $\Gamma_{\mathcal L}$ is generated by the projective classes of the ten pairwise differences among the five finite graph matrices $0,\  I_2,\  M_2,\  M_3,\  M_4.$

This matrix presentation depends on the choice of the distinguished pair $(L_0,L_\infty)$ and on the coordinates used to identify $K^4$ with $U\oplus W$. In the computations below, these choices are always fixed explicitly. Thus the group is understood through a normalized matrix presentation of the given configuration of skew lines.

\begin{example}[Lines on a smooth quadric]
Let $\mathcal Q\subset \PP^3$ be a smooth quadric surface. Then $\mathcal Q \simeq \PP^1\times \PP^1$, and the lines on $\mathcal Q$ belong to two rulings.
Two lines on $\mathcal Q$ are skew if and only if they belong to the same ruling.
Hence any finite configuration of pairwise skew lines on $\mathcal Q$ is contained in one ruling. 
Choose two of these lines as $L_0$ and $L_\infty$. After identifying $\PP^3=\PP(U\oplus W)$ with  $L_0=\PP(U),\  L_\infty=\PP(W),$  the remaining lines in the same ruling are graphs of scalar maps 
$\lambda I_2:U\longrightarrow W.$ 
Therefore every difference  $\lambda I_2-\mu I_2=(\lambda-\mu)I_2$  has trivial projective class in $\PGL_2(K)$. Consequently, $\Gamma_{\mathcal L}=\{1\}.$

Thus smooth quadrics yield only the trivial group. Smooth cubic surfaces are therefore the first natural class of surfaces in $\PP^3$ for which the associated group can exhibit nontrivial behavior.
\end{example}

\subsection{Sixers and blow-down models}

Let $S\subset \PP^3$ be a smooth cubic surface over an algebraically closed field.  A sixer on $S$ is a set
\[
        \mathcal L=\{L_1,\ldots,L_6\}
\]
of six pairwise skew lines on $S$.  The six lines can be contracted simultaneously. The contraction gives a birational morphism
\[
        S\longrightarrow \PP^2,
\]
and realizes $S$ as the blow-up of the plane at six points in general position.

Conversely, every presentation of $S$ as the blow-up of $\PP^2$ at six points in general position determines a sixer, namely the six exceptional curves.
Thus sixers are equivalent to blow-down models of $S$.

Classically, sixers occur as the two halves of Schl\"afli  double-sixes, \cite{Schlaefli}; in particular, the $36$ double-sixes give the $72$ sixers on a smooth cubic surface.  We will not use the double-six formalism below, but only the equivalent interpretation of sixers as blow-down models. Since every smooth cubic surface has $72$ sixers, each sixer determines a presentation of the surface as the blow-up of $\PP^2$ at six points.

For a fixed surface $S$, the automorphism group $\Aut(S)$ acts on the set of sixers.  The orbit decomposition measures the different symmetry types of blow-down models.  In this paper we study this action for the Fermat cubic surface.  We will see that its $72$ sixers split into two automorphism orbits, of sizes $72=54+18.$

The guiding question is whether the corresponding matrix groups $\Gamma_{\mathcal L}$ reflect the geometry of the associated blow-up models.

\subsection{The determinant square-class character}
\label{ss:determinant square-class}

Throughout the general construction, $K$ denotes a field.  In the Fermat computations the normalized matrices are defined over
\[
        K=\mathbb Q(\omega),\qquad \omega^2+\omega+1=0.
\]
The invariant considered below is attached to the projective representation over this field.  After extending scalars to $\mathbb C$, it becomes trivial, since $\mathbb C^*/(\mathbb C^*)^2=1.$
Thus the arithmetic information studied here belongs to the representation over its field of definition, rather than merely to the complex configuration obtained after scalar extension.

There is a natural homomorphism
\[
        \delta_K:\PGL_2(K)\longrightarrow
        K^*/(K^*)^2.
\]
If $g\in\PGL_2(K)$ and $\widetilde g\in\GL_2(K)$ is any lift, define
\[
        \delta_K(g)
        =
        [\det(\widetilde g)].
\]
This is well defined: replacing $\widetilde g$ by $\lambda\widetilde g$, with $\lambda\in K^*$, multiplies the determinant by $\lambda^2$, which does not change its square class.  Moreover, $\delta_K$ is a group homomorphism because the determinant is multiplicative. 

For a group $\Gamma$, its first homology with coefficients in
$\mathbb F_2$ is the mod-$2$ abelianization
\[
        H_1(\Gamma;\mathbb F_2)
        \cong
        \Gamma^{\mathrm{ab}}\otimes_{\mathbb Z}\mathbb F_2,
\]
where 
        $\Gamma^{\mathrm{ab}}
        =
        \Gamma/[\Gamma,\Gamma]$ 
and $[\Gamma,\Gamma]$ is the commutator subgroup. Equivalently,
\[
        H_1(\Gamma;\mathbb F_2)
        \cong
        \Gamma\big/\bigl([\Gamma,\Gamma]\Gamma^2\bigr),
\]
where $\Gamma^2$ denotes the subgroup generated by the squares $g^2$, $g\in\Gamma$.  Thus $H_1(\Gamma;\mathbb F_2)$ is an $\mathbb F_2$-vector space, and its dual is naturally identified with the space of homomorphisms
        $\Hom(\Gamma,\mathbb Z/2).$
In particular,
        $\dim_{\mathbb F_2}H_1(\Gamma;\mathbb F_2)$
measures the number of independent mod-$2$ characters of $\Gamma$.

Since $K^*/(K^*)^2$ is an $\mathbb F_2$-vector space, the
restriction of $\delta_K$ to a subgroup
$\Gamma\subset\PGL_2(K)$ factors through
       $ H_1(\Gamma;\mathbb F_2).$ 
Consequently,
\[
        \dim_{\mathbb F_2}H_1(\Gamma;\mathbb F_2)
        \ge
        \dim_{\mathbb F_2}\delta_K(\Gamma).
\]

In the Fermat computations below, this character is the main invariant used to compare the two sixer orbits.  Finite reductions modulo primes are then used as a computable shadow of the determinant square-class calculation.

\section{The Fermat cubic and its sixers}
\label{sec:fermat-computation}

Let  $S_F=\{x_0^3+x_1^3+x_2^3+x_3^3=0\}\subset\PP^3.$ 
We work over
\[
        K=\mathbb Q(\omega),\qquad \omega^2+\omega+1=0,
\]
so that the third roots of unity are   $1,\  \omega,\  \omega^2.$

The $27$ lines on the Fermat cubic admit a particularly simple explicit description.  Partition the four coordinates into two unordered pairs.  There are three such partitions:
\[
\pi_0=\{\{0,1\},\{2,3\}\},\qquad
\pi_1=\{\{0,2\},\{1,3\}\},\qquad
\pi_2=\{\{0,3\},\{1,2\}\}.
\]
For each of the three partitions, we use the ordering of the two pairs displayed above.  Thus, if
\[
        \pi=\{\{i,j\},\{k,\ell\}\},
\]
with the pairs written in this order, we define
\[
        L_{\pi;a,b}:
        \qquad
        x_i+a x_j=0,\qquad x_k+b x_{\ell}=0,
\]
for   $a,b\in\{1,\omega,\omega^2\}.$ 
Since \(a^3=b^3=1\), these equations imply
\[
        x_i^3+x_j^3=0,
        \qquad
        x_k^3+x_{\ell}^3=0,
\]
so that \(L_{\pi;a,b}\subset S_F\).
This gives 
       $ 3\cdot 3\cdot 3=27$ 
lines.  We label these lines by 
        $L_{\pi;a,b}.$ 

For example, for the partition 
        $\pi_0=\{\{0,1\},\{2,3\}\}$ 
and the choice $a=\omega^2$, $b=\omega$, the corresponding line is
\[
        L_{\pi_0;\omega^2,\omega}:
        \qquad
        x_0+\omega^2 x_1=0,\qquad x_2+\omega x_3=0.
\]
Let $\mathcal G_F$ be the skewness graph of the $27$ lines on $S_F$: its vertices are the lines, and two vertices are joined when the corresponding lines are skew.  A sixer on $S_F$ is exactly a clique of size $6$ in~$\mathcal G_F$.

The Python script in Appendix~\ref{app:fermat-computation} constructs the $27$ Fermat lines, records their numerical ordering in the notation $L_{\pi;a,b}$, builds the skewness graph $\mathcal G_F$, and enumerates its $6$-cliques.  Its output gives
\[
        |V(\mathcal G_F)|=27,\qquad |E(\mathcal G_F)|=216,
        \qquad \#\{\text{sixers}\}=72.
\]
Thus the computation recovers the classical number of sixers on a smooth cubic surface.

The automorphism group of the Fermat cubic is well known; see, for example,
\cite{Dolgachev,DolgachevDuncan}.  It is
\[
        \Aut(S_F)\cong \mu_3^3\rtimes S_4,
\]
where $S_4$ permutes the four coordinates, and $\mu_3^3$ is the diagonal subgroup acting by coordinatewise multiplication by third roots of unity, modulo the scalar diagonal.

The script realizes this action explicitly on the set of $27$ lines, and hence on the skewness graph $\mathcal G_F$ and on its $72$ sixers. For background on computational methods for finite and permutation groups, see, for instance, \cite{HoltEickOBrien}. The resulting finite permutation-group computation gives two automorphism orbits:
\[
        72=54+18.
\]
The complete lists of the two orbits are printed in Appendix~\ref{app:fermat-computation}.  With respect to the numerical ordering of the Fermat lines used there, representatives are
\[
        \mathcal L_{54}=\{0,4,10,12,20,25\},
        \qquad
        \mathcal L_{18}=\{0,4,8,10,14,15\}.
\]
Here the subscripts indicate the sizes of the corresponding automorphism
orbits.

Equivalently, in the notation $L_{\pi;a,b}$, these representatives are
\[
\begin{aligned}
\mathcal L_{54}=\{&
L_{\pi_0;1,1},\
L_{\pi_0;\omega,\omega},\
L_{\pi_1;1,\omega},\
L_{\pi_1;\omega,1},\
L_{\pi_2;1,\omega^2},\
L_{\pi_2;\omega^2,\omega}
\},
\end{aligned}
\]
and
\[
\begin{aligned}
\mathcal L_{18}=\{&
L_{\pi_0;1,1},\
L_{\pi_0;\omega,\omega},\
L_{\pi_0;\omega^2,\omega^2},\
L_{\pi_1;1,\omega},\
L_{\pi_1;\omega,\omega^2},\
L_{\pi_1;\omega^2,1}
\}.
\end{aligned}
\]
In the next section we give a geometric interpretation of these two orbit types through their corresponding plane blow-up models.
In Section~\ref{sec:fermat-associated-groups} we normalize these two representatives and compute the associated matrix groups.  

%---------------------------------------------------------------------
%---------------------------------------------------------------------

\section{The plane blow-up model and the geometry of the sixer orbits}\label{sec:blowup-interpretation}
We now give a geometric interpretation of the two $\operatorname{Aut}(S_F)$-orbits of sixers by means of plane blow-up models.
This will also identify their stabilizers with projective automorphism groups
of suitable configurations of six points.

\subsection{An explicit blow-up model}

Let
\begin{align*}
    p_1&=(1:0:0),\quad p_2=(0:1:0),\quad p_3=(0:0:1), \\
    p_4&=(1:1:1),\quad p_5=(1:\omega:\omega^2),\  p_6=(1:\omega^2:\omega)
\end{align*}
be points in $\PP^2_K$, where $\omega^2+\omega+1=0$. These points are in general position. Let
\[\pi\colon X=\Bl_{p_1,\ldots,p_6} \PP^2 \longrightarrow \PP^2\]
be their blow-up, let $H$ be the pull-back of the class of a line, and let $E_1,\ldots,E_6$ be the exceptional curves. The anticanonical class of $X$ is \(-K_X=3H-E_1-\cdots-E_6.\) The complete linear system $|-K_X|$ defines an embedding \( X\hookrightarrow \PP^3\). With a suitable choice of basis of the cubic forms vanishing at $p_1,\ldots,p_6$, its image is the cubic surface
\[S\colon y^2z-yz^2-x^2w+xw^2 =0.\]
The surface $S$ is projectively equivalent to the Fermat cubic. Indeed, set
\[
x_0= -\omega y-\omega^2 z, \qquad
x_1= y+\omega^2 z, \qquad
x_2= \omega x+\omega^2 w, \qquad
x_3= -x-\omega^2 w.
\]
The corresponding linear transformation is invertible and
\[x_0^3+x_1^3+x_2^3+x_3^3 = -3(2\omega+1) \bigl(y^2z-yz^2-x^2w+xw^2\bigr).\]
It therefore identifies $S$ with
\[S_F=\{x_0^3+x_1^3+x_2^3+x_3^3=0\}.\]
We use this identification to translate the standard blow-up model of the $27$ lines on $X$ into the Fermat-line notation introduced in Section~\ref{sec:fermat-computation}.

\subsection{The twenty-seven lines in the blow-up model}

The $27$ lines on $X$ are represented by the following divisor classes:
\begin{align*}
E_i,  & &\qquad 1 \leq i \leq 6, \\
L_{ij} &= H - E_i - E_j, &\qquad 1 \leq i < j \leq 6, \\
Q_i    &= 2H - \sum_{j \neq i} E_j, &\qquad 1 \leq i \leq 6.
\end{align*}
Geometrically, the curves $E_i$ are the exceptional divisors, the curves $L_{ij}$ are the strict transforms of the lines through $p_i$ and $p_j$, and the curves $Q_i$ are the strict transforms of the conics through the five points different from $p_i$. Each of these curves satisfies
\[C^2=-1,\qquad (-K_X)\cdot C=1,\]
and is therefore mapped to a line by the anticanonical embedding. Thus the $27$ lines consist of
\[ 6\text{ curves }E_i,\qquad 15 \text{ curves }L_{ij}, \qquad 6\text{ curves }Q_i. \]
Using the projective equivalence above, we computed the correspondence between these curves and the Fermat lines $L_{\pi;a,b}$ introduced in Section \ref{sec:fermat-computation}; see Appendix~\ref{app:singular-blowup}. In particular, the representative \(\mathcal L_{18}= \{0,4,8,10,14,15\}\)
corresponds to
\[\mathcal L_{18}= \{E_3,Q_3,L_{12},L_{24},L_{25},L_{26}\},\]
whereas \(\mathcal L_{54}= \{0,4,10,12,20,25\}\)
corresponds to
\[\mathcal L_{54}= \{L_{12},L_{16},L_{26},Q_3,Q_4,Q_5\}.\]
The first representative is of type $E+Q+4L$, while the second is of type $3L+3Q$.

\subsection{Sixers in the fixed blow-up model}

The intersection relations among the $27$ curves are determined by
\[H^2=1, \qquad  H\cdot E_i=0, \qquad  E_i\cdot E_j =-\delta_{ij}.\]
In particular,
\[E_i \cdot Q_j=
	\begin{cases}
	0,&i=j,\\
	1,&i\neq j,
	\end{cases}
\qquad
E_i\cdot L_{jk}=
	\begin{cases}
	1,&i\in\{j,k\},\\
	0,&i\notin\{j,k\},
	\end{cases}
\]
and, for distinct curves,
\[L_{ij} \cdot L_{k\ell} =0 \quad\Longleftrightarrow\quad |\{i,j\} \cap \{k,\ell\}|=1,\]
while
\[L_{ij} \cdot Q_k=0   \quad\Longleftrightarrow\quad k \notin\{i,j\}.\]
These relations give the following description of all sixers in the fixed blow-up model.

\begin{proposition} For $I\subset\{1,\ldots,6\}$, set   $I^c=\{1,\ldots,6\}\setminus I$. 
The $72$ sixers on $X$ are of the following five types:
\begin{align*}
\mathcal E &=\{E_1,\ldots,E_6\},\\
\mathcal Q &=\{Q_1,\ldots,Q_6\},\\
\mathcal S_I^{E} &=\{E_i:i\in I\}\cup \{L_{ab}:a,b\in I^c,\ a<b\}, && |I|=3,\\
\mathcal S_I^{Q} &=\{L_{ab}:a,b\in I,\ a<b\}\cup \{Q_j:j\in I^c\}, && |I|=3,\\
\mathcal S_{i,k} &=\{E_i,Q_i\}\cup \{L_{kj}:j\notin\{i,k\}\}, && i\neq k.
\end{align*}
\end{proposition}

\begin{proof}
The intersection relations above show that every displayed set consists of six pairwise disjoint $(-1)$-curves. The numbers of such sets are respectively \(1, 1,  \binom{6}{3}, \binom{6}{3}, 6\cdot5.\) Their total number is \(1+1+20+20+30=72.\) Since a smooth cubic surface has exactly $72$ sixers, the list is complete.
\end{proof}

\subsection{Two plane blow-down models}

The exceptional sixer \(\mathcal E=\{E_1,\ldots,E_6\}\) belongs to the orbit of size \(18\). Its contraction is the original blow-down \(\pi \colon X\longrightarrow \PP^2,\)
and the corresponding configuration of points is
\[Z_{18}=\left\{
\begin{aligned}
&(1:0:0),\ (0:1:0),\ (0:0:1),\ (1:1:1),\\
&(1:\omega:\omega^2),\ (1:\omega^2:\omega)
\end{aligned}
\right\}.
\]
For the orbit of size \(54\), we use the sixer \(\mathcal S_{54}=\{L_{12},L_{16},L_{26},Q_3,Q_4,Q_5\}.\) The pull-back of the class of a line under the corresponding blow-down is \(D_{54}=4H-2E_1-2E_2-E_3-E_4-E_5-2E_6.\)
Indeed,
\[D_{54}^2=1, \qquad  (-K_X)\cdot D_{54}=3,
\]
and
\[
D_{54}\cdot C=0 \qquad \text{for every } C\in\mathcal S_{54}.
\]
The linear system \(|D_{54}|\) is represented in the original plane by the quartics having double points at \(p_1,p_2,p_6\) and passing through \(p_3,p_4,p_5\). The Singular computation described in Appendix~\ref{app:singular-blowup} verifies that this system defines a morphism \(\beta_{54}\colon X\longrightarrow\PP^2\) which contracts precisely the six curves in \(\mathcal S_{54}\).
The six image points obtained by this contraction are, after reordering the first three points and applying a diagonal projective change of coordinates, projectively equivalent to
\[
Z_{54}=\left\{
\begin{aligned}
&(1:0:0),\ (0:1:0),\ (0:0:1),\ (1:1:1),\\
&(\omega+2:1-\omega:3),\ (1-\omega:\omega+2:3)
\end{aligned}
\right\}.
\]
Thus \(Z_{18}\) and \(Z_{54}\) give two explicit plane blow-up models of the Fermat cubic surface.

\subsection{Verification of the second blow-up model}

Let \(f_{18}=y^2z-yz^2-x^2w+xw^2\) be the cubic equation obtained from \(Z_{18}\). Blowing up the six points of \(Z_{54}\) and applying the anticanonical map gives the cubic surface \(S_{54}=\{f_{54}=0\}\subset\PP^3,\)
where
\begin{equation*}
f_{54}=
3x^2y-3\omega xy^2+3x^2z-(3\omega+3)y^2z-\omega xz^2-(\omega+1)yz^2
-3\omega x^2w-2yzw+xw^2.
\end{equation*}
Consider the matrix
\[M=
\begin{pmatrix}
-\omega&0&\omega+1&0\\
-1&\omega&-\omega&1\\
-2&-\omega&-\omega&1\\
4\omega+4&1&-2&-2\omega-2
\end{pmatrix}.
\]
Its determinant is \(\det(M)=-3(\omega+1)\neq0,\) and a direct computation gives
\[
f_{54}\bigl(M(x,y,z,w)^{\mathsf T}\bigr)=-3(\omega+1)f_{18}(x,y,z,w).
\]
Consequently, \(M\) induces a projective equivalence between the two
anticanonical models.

Let
\[\tau_M\colon \PP^3\longrightarrow\PP^3,  \qquad [X]\longmapsto[MX]\]
be the corresponding projective transformation, and let \(E'_1,\ldots,E'_6\) denote the exceptional curves in the blow-up of \(Z_{54}\). The Singular computation described in Appendix~\ref{app:singular-blowup} gives
\[
\begin{aligned}
\tau_M^{-1}(E'_1)&=L_{12},&
\tau_M^{-1}(E'_2)&=L_{16},&
\tau_M^{-1}(E'_3)&=L_{26},\\
\tau_M^{-1}(E'_4)&=Q_3,&
\tau_M^{-1}(E'_5)&=Q_4,&
\tau_M^{-1}(E'_6)&=Q_5.
\end{aligned}
\]
This verifies that the second plane model is obtained by contracting exactly the chosen sixer~\(\mathcal S_{54}\).

\subsection{Plane automorphisms and the two orbit sizes}

For a finite configuration \(Z\subset\PP^2\), set
\[
\Aut_{\PP^2}(Z)=\{g\in\PGL_3(K):g(Z)=Z\}.
\]
\begin{proposition}\label{prop:Stab=Aut}
Let \(\mathcal L\) be a sixer on \(S_F\), and let \(Z_{\mathcal L}\subset\PP^2\) be the six-point configuration obtained by contracting \(\mathcal L\). Then
\[
\operatorname{Stab}_{\Aut(S_F)}(\mathcal L) \cong \Aut_{\PP^2}(Z_{\mathcal L}).
\]
\end{proposition}
\begin{proof}
An automorphism of \(S_F\) preserving \(\mathcal L\) descends through the contraction of the six curves to a projective automorphism of \(\PP^2\) preserving the six image points.

Conversely, a projective automorphism of \(\PP^2\) preserving \(Z_{\mathcal L}\) lifts to an automorphism of the blow-up \(\Bl_{Z_{\mathcal L}}\PP^2\). Since the anticanonical map identifies this blow-up with \(S_F\), the lifted automorphism induces an automorphism of \(S_F\) preserving $\mathcal L$. 
\end{proof}

The projective automorphism groups of the two point configurations were computed by testing all \(6!=720\) permutations of their points; see Appendix~\ref{app:singular-blowup}. For each permutation, the computation determines whether it is induced by an element of \(\PGL_3(K)\). The results are
\[|\Aut_{\PP^2} (Z_{18})| =36, \qquad  |\Aut_{\PP^2}(Z_{54})| =12. \]
Since
\[|\Aut(S_F)|=|\mu_3^3\rtimes S_4|=3^3\cdot24=648,\]
the orbit--stabilizer theorem gives
\[|\Aut(S_F)\cdot\mathcal E|=\frac{648}{36}=18, \quad \text{ and }
\quad |\Aut(S_F)\cdot\mathcal S_{54}|=\frac{648}{12}=54.\]
Thus the two automorphism orbits of sixers admit a direct plane-geometric interpretation: the corresponding six-point configurations have projective automorphism groups of different orders. The orbit of size \(18\) corresponds to the more symmetric configuration \(Z_{18}\), whereas the orbit of size \(54\) corresponds to \(Z_{54}\).

We summarize the geometric picture obtained in this section in the following theorem.
\begin{theorem}
The $72$ sixers on the Fermat cubic surface split into two $\Aut(S_F)$-orbits of sizes $18$ and $54$. Representatives of the two orbits admit plane blow-down models with six-point configurations $Z_{18}$ and $Z_{54}$ satisfying
\[
|\Aut_{\PP^2}(Z_{18})|=36,
\qquad
|\Aut_{\PP^2}(Z_{54})|=12.
\]
Moreover, these projective automorphism groups identify with the stabilizers of the corresponding sixers.
\end{theorem}
\begin{proof}
The two orbit representatives and the configurations $Z_{18}$ and $Z_{54}$ are constructed above. Proposition~\ref{prop:Stab=Aut} identifies their projective automorphism groups with the corresponding stabilizers, while the Singular computations give orders $36$ and $12$. Since $|\Aut(S_F)|=648$, the orbit--stabilizer theorem gives the orbit sizes $18$ and $54$.
\end{proof}

In the following sections, we compare this geometric distinction with the associated projective matrix groups \(\Gamma_{18}\) and \(\Gamma_{54}\).

%---------------------------------------------------------------------
%---------------------------------------------------------------------

\section{The associated groups for the Fermat orbits}
\label{sec:fermat-associated-groups}
We now apply the matrix normalization described in Section~2 to the two representative sixers $\mathcal L_{54}$ and $\mathcal L_{18}$. 
The computation is performed by the Python script in Appendix~\ref{app:fermat-normalization}.  For the representative $\mathcal L_{54}=(0,4,10,12,20,25),$ we choose the ordered triple $(0,4,10)$, whereas for  $\mathcal L_{18}=(0,4,8,10,14,15)$ we choose $(0,4,8)$.

For each representative, the chosen ordered triple is sent, by a projective change of coordinates over $K$, to the standard triple
\[
\begin{aligned}
L_0 :\quad z=w=0,\qquad 
L_\infty :\quad x=y=0,\qquad 
L_1 :\quad x-z= y-w=0.
\end{aligned}
\]
Equivalently, with respect to the decomposition $K^4=U\oplus W$, these three lines are represented by $ 0,\  \infty,\  I_2.$ The remaining three lines are represented by matrices $M_2,M_3,M_4$, so that
the normalized sixer has the form
\[
        0,\quad \infty,\quad I_2,\quad M_2,\quad M_3,\quad M_4.
\]
The matrices are defined over
\[
        K=\mathbb Q(\omega),\qquad \omega^2+\omega+1=0.
\]
The corresponding subgroups of $\PGL_2(K)$ are denoted by $\Gamma_{54}$ and 
        $\Gamma_{18},$ 
according to the automorphism orbit of the original sixer.
\subsection{The orbit of size $54$}

For the representative $\mathcal L_{54}$, the normalization gives
\[
M_2=
\begin{pmatrix}
0&1+\omega\\
\omega&-1
\end{pmatrix},
\qquad
M_3=
\begin{pmatrix}
-\omega&\omega\\
0&-1-\omega
\end{pmatrix},
\qquad
M_4=
\begin{pmatrix}
1+\omega&1\\
0&\omega
\end{pmatrix}.
\]
Thus
\[
        \Gamma_{54}
        =
        \left\langle
        [A-B]\ :\
        A,B\in\{0,I_2,M_2,M_3,M_4\},\ A\neq B
        \right\rangle
        \subseteq \PGL_2(K).
\]
\subsection{The orbit of size $18$}

For the representative $\mathcal L_{18}$, the normalization gives
\[
M_2=
\begin{pmatrix}
0&-1-\omega\\
-\omega&1
\end{pmatrix},
\qquad
M_3=
\begin{pmatrix}
0&1\\
-1&1
\end{pmatrix},
\qquad
M_4=
\begin{pmatrix}
0&\omega\\
1+\omega&1
\end{pmatrix}.
\]
Thus
\[
        \Gamma_{18}
        =
        \left\langle
        [A-B]\ :\
        A,B\in\{0,I_2,M_2,M_3,M_4\},\ A\neq B
        \right\rangle
        \subseteq \PGL_2(K).
\]
The next section compares the two normalized projective representations through their determinant square-class images over $K$ and then relates this arithmetic distinction to their finite reductions.

\section{Determinant square classes}
\label{sec:determinant-square-classes}

In this section, we apply the determinant square-class character $\delta_K$, defined in Section~\ref{ss:determinant square-class}, to the matrix groups $\Gamma_{54}$ and $\Gamma_{18}$ associated with the two Fermat orbits.  We work over
\[
        K=\mathbb Q(\omega),\qquad \omega^2+\omega+1=0,
\]
and regard $K^*/(K^*)^2$ as an $\mathbb F_2$-vector space, with group law induced by multiplication in $K^*$.

\subsection{The two Fermat orbits}

The generators of each group are the projective classes of the ten pairwise differences among the corresponding matrices 
        $0,\quad I_2,\quad M_2,\quad M_3,\quad M_4$
from the previous section.  It is therefore enough to compute the determinant square classes of these ten differences.

The raw determinant data are computed by the Python script in Appendix~\ref{app:fermat-determinants}.  Starting from the normalized matrices of Appendix~\ref{app:fermat-normalization}, the script computes
\[
        \det(A-B),
        \qquad
        A,B\in\{0,I_2,M_2,M_3,M_4\},
\]
for all ten pairwise differences.

For the orbit of size $18$, the output of Appendix~\ref{app:fermat-determinants} gives, up to repetition, the determinant values
        $\{1,3\}.$
Hence
\[
        \delta_K(\Gamma_{18})
        =
        \langle [3]\rangle
        \subset K^*/(K^*)^2.
\]
For the orbit of size $54$, the corresponding determinant values are
\[
        \{-3,-1-2\omega,-1,1,1+2\omega,3\}
        =
        \{\pm1,\ \pm3,\ \pm(1+2\omega)\}.
\]
The signs do not introduce any further square classes.  Indeed, in $K=\mathbb Q(\omega)$ we have
\[
        (1+2\omega)^2
        =
        1+4\omega+4\omega^2
        =
        1+4\omega+4(-1-\omega)
        =
        -3.
\]
Thus $-3$ is a square in $K$, and hence
        $[-3]=[1]$ in  $K^*/(K^*)^2.$
Multiplying this identity by $[3]$, we obtain
        $[-1]=[3].$ 
Therefore the classes of the negative determinants are already contained in the subgroup generated by $[3]$ and $[1+2\omega]$.  Since both $3$ and $1+2\omega$ occur among the determinant values, it follows that
\[
        \delta_K(\Gamma_{54})
        =
        \langle [3],[1+2\omega]\rangle.
\]
We now check that the two classes $[3]$ and $[1+2\omega]$ are nontrivial and independent.

\begin{lemma}
Let $K=\mathbb Q(\omega)$, where $\omega^2+\omega+1=0$. Then the classes $[3]$ and $[1+2\omega]$ are independent in $K^*/(K^*)^2$.
\end{lemma}

\begin{proof}
The group $K^*/(K^*)^2$ is an $\mathbb F_2$-vector space, since every class has order dividing $2$. Hence two classes $[u]$ and $[v]$ are independent if and only if none of the three nontrivial combinations
\[
[u],\qquad [v],\qquad [uv]
\]
is trivial. Therefore, it is enough to show that none of
\[
3,\qquad 1+2\omega,\qquad 3(1+2\omega)
\]
is a square in $K$.

Since $\omega$ and $\omega^2$ are the two roots of the minimal polynomial $t^2+t+1$ over $\mathbb Q$, the assignment $\omega\mapsto\omega^2$
extends to the nontrivial $\mathbb Q$-automorphism $\sigma:K\to K$. Let $N(x)=x\sigma(x)$ be the norm from $K$ to
$\mathbb Q$. For $a,b\in\mathbb Q$, we have
\[
N(a+b\omega)
=(a+b\omega)(a+b\omega^2)
=a^2-ab+b^2.
\]
In particular, if $x=y^2$ is a square in $K$, then $N(x)=N(y)^2$ is a square in $\mathbb Q$.

We first show that $3$ is not a square in $K$. Suppose that
$3=(a+b\omega)^2$ for some $a,b\in\mathbb Q$. Since
$\omega^2=-1-\omega$, we obtain 
$(a+b\omega)^2
=a^2-b^2+(2ab-b^2)\omega.$ 
Comparing the coefficient of $\omega$ gives $b(2a-b)=0$. If $b=0$,
then $a^2=3$, which is impossible in $\mathbb Q$. If $b=2a$, then
$(a+2a\omega)^2=-3a^2$, so $-3a^2=3$, again impossible. Hence $3$ is
not a square in $K$. 

Next,  $N(1+2\omega)=1-2+4=3.$ 
Since $3$ is not a square in $\mathbb Q$, it follows that $1+2\omega$ is not a square in $K$.

Finally, by multiplicativity of the norm,
\[
N\bigl(3(1+2\omega)\bigr)
=N(3)N(1+2\omega)
=9\cdot 3
=27.
\]
Since $27$ is not a square in $\mathbb Q$, the element
$3(1+2\omega)$ cannot be a square in $K$.

Thus neither $[3]$, nor $[1+2\omega]$, nor their product is trivial in $K^*/(K^*)^2$. Hence $[3]$ and $[1+2\omega]$ are independent.
\end{proof}
\begin{theorem}
For the two groups associated with the Fermat orbits, we have
\[
        \delta_K(\Gamma_{18})\cong \mathbb Z/2,
\qquad \text{and}\qquad 
        \delta_K(\Gamma_{54})\cong (\mathbb Z/2)^2.
\]
Consequently,
\[
        \dim_{\mathbb F_2}H_1(\Gamma_{18};\mathbb F_2)\geq 1,
\qquad
\text{and}
\qquad
        \dim_{\mathbb F_2}H_1(\Gamma_{54};\mathbb F_2)\geq 2.
\]
\end{theorem}

\begin{proof}
The computations above show that
$\delta_K(\Gamma_{18})=\langle[3]\rangle$ and $\delta_K(\Gamma_{54})=\langle[3],[1+2\omega]\rangle.$ 
By the lemma, the class $[3]$ is nontrivial, while $[3]$ and
$[1+2\omega]$ are independent. Therefore
$\delta_K(\Gamma_{18})\cong \mathbb Z/2$
 and 
$\delta_K(\Gamma_{54})\cong (\mathbb Z/2)^2$.

Since $K^*/(K^*)^2$ is an elementary abelian $2$-group, the restriction of
$\delta_K$ to either $\Gamma\in\{\Gamma_{18},\Gamma_{54}\}$
vanishes on commutators and squares.
Hence it factors through 
$\Gamma/[\Gamma,\Gamma]\Gamma^2
\cong H_1(\Gamma;\mathbb F_2).$
Thus $\delta_K(\Gamma)$ is a quotient of
$H_1(\Gamma;\mathbb F_2)$, and therefore
$\dim_{\mathbb F_2}H_1(\Gamma;\mathbb F_2)
\geq \dim_{\mathbb F_2}\delta_K(\Gamma).$
The stated inequalities follow.
\end{proof}

\section{Distinguishing the two Fermat orbits}
\label{sec:distinguishing-fermat-orbits}

The preceding section shows that the two Fermat orbit types give normalized projective representations with different determinant square-class images:
\[
        \delta_K(\Gamma_{18})
        \cong \mathbb Z/2,
\qquad \text{and} \qquad
        \delta_K(\Gamma_{54})
        \cong (\mathbb Z/2)^2.
\]
In this section we compare this arithmetic distinction with the finite reductions of the two representations and verify that the resulting modulo-$13$ behavior is independent of the chosen normalization.

\subsection{Finite images modulo primes}

The determinant square-class distinction has a visible congruence shadow.  We compute the finite images of the two normalized representations modulo several primes $p\equiv 1\pmod 3.$ For such a prime, the polynomial
        $t^2+t+1$ 
splits over $\mathbb F_p$.  We choose an element
$r\in\mathbb F_p$ satisfying 
        $r^2+r+1=0$ 
and reduce the matrix entries by sending
\[
        \omega\longmapsto r.
\]
This gives finite subgroups
\[
        \overline{\Gamma}_{54}^{(p)},
        \quad
        \overline{\Gamma}_{18}^{(p)}
        \subseteq \PGL_2(\mathbb F_p).
\]
The computation in Appendix~\ref{app:fermat-finite-images} records the orders of these finite images.  The generators are obtained by first reducing
\[
        0,\quad I_2,\quad M_2,\quad M_3,\quad M_4
\]
modulo $p$, then forming their ten pairwise differences, and finally taking the corresponding projective classes.  In particular, the individual matrices $M_i$ are not normalized projectively before the differences are
formed.

\[
\begin{array}{c|c|cc|c}
p & r
& |\overline{\Gamma}_{54}^{(p)}|
& |\overline{\Gamma}_{18}^{(p)}|
& |\PGL_2(\mathbb F_p)|
\\
\hline
7  & 2  & 336   & 336   & 336\\
13 & 3  & 2184  & 1092  & 2184\\
19 & 7  & 6840  & 6840  & 6840\\
31 & 5  & 29760 & 29760 & 29760\\
37 & 10 & 25308 & 25308 & 50616\\
43 & 6  & 79464 & 79464 & 79464.
\end{array}
\]

For odd $p$, we have 
        $|\PGL_2(\mathbb F_p)|=p(p^2-1),
        \ 
        |\PSL_2(\mathbb F_p)|=\frac{p(p^2-1)}{2}.$ 
        
Combining these orders with the determinant square classes of the generators, the finite images may be identified as follows:
\[
\begin{array}{c|cc}
p
& \overline{\Gamma}_{54}^{(p)}
& \overline{\Gamma}_{18}^{(p)}
\\
\hline
7  & \PGL_2(\mathbb F_7)    & \PGL_2(\mathbb F_7)\\
13 & \PGL_2(\mathbb F_{13}) & \PSL_2(\mathbb F_{13})\\
19 & \PGL_2(\mathbb F_{19}) & \PGL_2(\mathbb F_{19})\\
31 & \PGL_2(\mathbb F_{31}) & \PGL_2(\mathbb F_{31})\\
37 & \PSL_2(\mathbb F_{37}) & \PSL_2(\mathbb F_{37})\\
43 & \PGL_2(\mathbb F_{43}) & \PGL_2(\mathbb F_{43}).
\end{array}
\]
The prime $13$ is especially useful for the present comparison.  With the choice
\[
        \omega\longmapsto 3\in\mathbb F_{13},
\]
the nonzero squares in $\mathbb F_{13}$ are $\{1,3,4,9,10,12\}.$
In particular, $ 3=4^2$ in $\mathbb F_{13},$ so the class $[3]$ becomes trivial in 
        $\mathbb F_{13}^{*}/(\mathbb F_{13}^{*})^2.$
 
On the other hand,
\[
        1+2\omega\longmapsto 1+2\cdot3=7,
\qquad
\text{and}
\qquad
        7\notin\{1,3,4,9,10,12\}.
\]
Hence $[1+2\omega]$ maps to the nontrivial square class modulo $13$.

Therefore the determinant character modulo $13$ is trivial for the orbit of size $18$, while it is nontrivial for the orbit of size $54$.  The finite
image computation sharpens this observation:
\[
        \overline{\Gamma}_{18}^{(13)}
        =
        \PSL_2(\mathbb F_{13}),
        \qquad
        \overline{\Gamma}_{54}^{(13)}
        =
        \PGL_2(\mathbb F_{13}).
\]

This congruence computation is supplementary to the determinant square-class argument over $K=\mathbb Q(\omega).$  The square-class computation is independent of the choice of a prime, whereas the finite images modulo $p$ depend on the chosen reduction.  Nevertheless, the modulo $13$ image gives a concrete finite field shadow of the distinction between the two Fermat orbit types.

\subsection{Independence of the normalization triple}

The previous computation uses one normalization for each orbit representative. As a further check, Appendix~\ref{app:fermat-mod13} runs through all $6\cdot5\cdot4=120$ ordered choices of the normalization triple
        $(L_0,L_\infty,L_1)$
inside each representative sixer.

The output is summarized in the following table:
\[
\begin{array}{c|c|c}
\text{orbit}
& \begin{array}{c}
     \text{number of nonsquare determinants} \\
     \text{among the ten generators}
\end{array}
& \text{number of normalizations}
\\
\hline
\mathcal O_{54} & 4 & 72\\
\mathcal O_{54} & 6 & 48\\
\mathcal O_{18} & 0 & 120.
\end{array}
\]
Thus every normalization of the representative in $\mathcal O_{18}$ has trivial determinant character modulo $13$, whereas every normalization of the representative in $\mathcal O_{54}$ has nontrivial determinant character modulo $13$. This shows that the modulo-$13$ distinction is not an artifact of the chosen ordered triple $(L_0,L_\infty,L_1)$.

\subsection{What is distinguished}

The computations above distinguish the projective representations associated with the two Fermat orbit types. In particular, their determinant square-class images are different:
\[
        \delta_K(\Gamma_{18})\cong\mathbb Z/2,
        \qquad
        \delta_K(\Gamma_{54})\cong(\mathbb Z/2)^2.
\]
However, these computations do not by themselves imply that 
        $\Gamma_{18}\not\cong\Gamma_{54}$ 
as abstract groups.  They distinguish the arithmetic projective representations over $K=\mathbb Q(\omega)$ and their finite reductions modulo primes, but the abstract isomorphism problem for the two groups remains open.

\section{Discussion and open questions}
\label{sec:discussion}

The computations in this paper show that the two automorphism orbits of sixers on the Fermat cubic give projective representations with different determinant square-class images over $K=\mathbb Q(\omega)$. This distinction is arithmetic: it belongs to the representations over their field of definition, and the determinant square-class invariant becomes trivial after extending scalars to $\mathbb C$. The determinant square-class character also gives the lower bounds  $\dim_{\mathbb F_2}H_1(\Gamma_{18};\mathbb F_2)\ge~1,$ $\dim_{\mathbb F_2}H_1(\Gamma_{54};\mathbb F_2)\ge 2.$

However, it is not a complete invariant of the abstract groups.  In
particular, the computations do not by themselves imply that
 $\Gamma_{18}\not\cong\Gamma_{54}.$
This leads to the following question.

\begin{question}
Is  $\dim_{\mathbb F_2}H_1(\Gamma_{18};\mathbb F_2)=1?$
Equivalently, is the determinant square-class character the only nontrivial
homomorphism  $\Gamma_{18}\longrightarrow\mathbb Z/2?$
\end{question}

A positive answer would imply that
       $ \Gamma_{18}\not\cong\Gamma_{54}$
as abstract groups, since $\Gamma_{54}$ has at least two independent
mod-$2$ characters.

One possible approach is to compute presentations for $\Gamma_{18}$ and
$\Gamma_{54}$, or for suitable finite-index subgroups of ambient arithmetic
groups, and then compute their mod-$2$ abelianizations; see, for instance, \cite{MaclachlanReid, Sims}. 

A second direction concerns the relation between the arithmetic and geometric distinctions obtained in this paper. The plane blow-up models give a direct geometric explanation of the two orbit sizes: the corresponding six-point configurations have projective automorphism groups of orders \(36\) and \(12\).
It remains to understand whether this difference in plane symmetry is conceptually related to the determinant square-class distinction between \(\Gamma_{18}\) and \(\Gamma_{54}\).

Thus the paper compares three levels of structure:
\[
        \text{sixer orbit},
        \qquad
        \text{associated projective group},
        \qquad
        \text{plane blow-up model}.
\]
The orbit decomposition is explained geometrically by the automorphism groups of the corresponding six-point configurations, while the associated projective groups are distinguished arithmetically by their determinant square-class images. The remaining problems are to determine whether the associated groups are abstractly non-isomorphic and to clarify the relation between the geometric and arithmetic distinctions.

\section{Declarations}

\paragraph{\bf Reproducibility.}
The computations reported in this paper were carried out in Python and
\textsc{Singular}~\cite{Singular}. We use Python for the enumeration of sixers, the automorphism orbit decomposition, the normalization of orbit representatives, the determinant square-class computations, the finite-image computations, and the modulo $13$ normalization check. Each script is accompanied by the corresponding raw output. The \textsc{Singular} computations discussed in Section \ref{sec:blowup-interpretation} construct the two plane blow-up models, verify their projective equivalence, and compute the projective automorphism groups of the corresponding six-point configurations.

\paragraph{\bf Acknowledgments}
Favacchio is a member of GNSAGA-INdAM. 

\paragraph{\bf Funding}
The work of Favacchio was supported by the funding PREMIO\_SINGOLI\_RIC\_[2025] from the Department of Engineering, University of Palermo.

\paragraph{\bf Code availability.}
A public repository containing the source files is available at \cite{FermatSixersCode}.

\paragraph{\bf Declaration of interest statement}
No potential competing interest was reported by the authors.

\paragraph{\bf Generative AI disclosure.}
 The authors used ChatGPT (OpenAI) as an assistive tool for language editing and for limited assistance with debugging, refactoring, and improving the readability of parts of the Python code. All mathematical statements, proofs, algorithms, computational outputs, and interpretations were reviewed and verified by the authors, who take full responsibility for the content and reproducibility of the work.

\appendix

\section{Fermat sixer computation}
\label{app:fermat-computation}
The script \texttt{fermat\_sixers.py} is the source of the numerical statements in
Section~\ref{sec:fermat-computation}. The script constructs the $27$ lines on the Fermat cubic, determines the skew pairs
and the sixers, constructs the induced action of
$\operatorname{Aut}(S_F)$, and computes its orbits on the set of sixers.

\subsection*{Output} The relevant output is:
\begin{verbatim}
number of lines = 27
number of skew pairs = 216
number of sixers = 72
number of automorphisms = 648
orbit sizes = [54, 18]
\end{verbatim}

The two orbit representatives used in the subsequent computations are
\begin{verbatim}
orbit size 54 representative (0, 4, 10, 12, 20, 25)
orbit size 18 representative (0, 4, 8, 10, 14, 15)
\end{verbatim}

The script also prints the complete list of sixers in each orbit; we
record here only the orbit sizes and the representatives used in the
subsequent computations.

\section{Normalization of the Fermat representatives}
\label{app:fermat-normalization}

The script \texttt{fermat\_normalization.py} normalizes the two
representative sixers obtained in Appendix~\ref{app:fermat-computation},
\[
        \mathcal L_{54}=\{0,4,10,12,20,25\},
        \qquad
        \mathcal L_{18}=\{0,4,8,10,14,15\}.
\]
For each sixer, we choose the first three lines as $L_0,\ L_\infty,\ L_1.$ The remaining lines are written as graphs of linear maps, and the
normalization $M_i\mapsto M_1^{-1}M_i$ is applied so that $L_1$
corresponds to $I_2$.

\subsection*{Output}

Running the script gives the following output:
\begin{multicols}{2}
\begin{verbatim}
orbit 54 representative:
(0, 4, 10, 12, 20, 25)
chosen normalization:
L0 = 0 Linf = 4 L1 = 10

M2 =
[0, 1+omega]
[omega, -1]

M3 =
[-omega, omega]
[0, -1-omega]

M4 =
[1+omega, 1]
[0, omega]
orbit 18 representative:
(0, 4, 8, 10, 14, 15)
chosen normalization:
L0 = 0 Linf = 4 L1 = 8

M2 =
[0, -1-omega]
[-omega, 1]

M3 =
[0, 1]
[-1, 1]

M4 =
[0, omega]
[1+omega, 1]
\end{verbatim}
\end{multicols}
These are the normalized matrices used in
Appendix~\ref{app:fermat-determinants}.

\section{Finite reductions of the Fermat associated groups}
\label{app:fermat-finite-images}

The script \texttt{fermat\_finite\_reductions.py} computes the orders of the finite images of the two normalized Fermat representations modulo primes $p\equiv 1 \pmod 3.$

For each such prime we choose an element $r\in \mathbb F_p$ satisfying 
$r^2+r+1=0$ 
and reduce by sending $\omega\mapsto r$. The generators are obtained from the ten pairwise differences among  $0,\  I_2,\  M_2,\  M_3,\  M_4.$ The individual matrices $M_i$ are not normalized projectively before taking
differences.  The projective normalization is applied only to the resulting
generators $A-B$ in $\PGL_2(\mathbb F_p)$.
\subsection*{Output} The output gives
\[
\begin{array}{c|c|cc|c}
p & r
& |\overline{\Gamma}_{54}^{(p)}|
& |\overline{\Gamma}_{18}^{(p)}|
& |\PGL_2(\mathbb F_p)|
\\
\hline
7  & 2  & 336   & 336   & 336\\
13 & 3  & 2184  & 1092  & 2184\\
19 & 7  & 6840  & 6840  & 6840\\
31 & 5  & 29760 & 29760 & 29760\\
37 & 10 & 25308 & 25308 & 50616\\
43 & 6  & 79464 & 79464 & 79464.
\end{array}
\]
For odd $p$,  $|\PGL_2(\mathbb F_p)|=p(p^2-1),
        \quad
        |\PSL_2(\mathbb F_p)|=\frac{p(p^2-1)}{2}.$
Combining these orders with the determinant square classes of the generators, for the primes tested above the finite images are
\[
\begin{array}{c|cc}
p
& \overline{\Gamma}_{54}^{(p)}
& \overline{\Gamma}_{18}^{(p)}
\\
\hline
7  & \PGL_2(\mathbb F_7)    & \PGL_2(\mathbb F_7)\\
13 & \PGL_2(\mathbb F_{13}) & \PSL_2(\mathbb F_{13})\\
19 & \PGL_2(\mathbb F_{19}) & \PGL_2(\mathbb F_{19})\\
31 & \PGL_2(\mathbb F_{31}) & \PGL_2(\mathbb F_{31})\\
37 & \PSL_2(\mathbb F_{37}) & \PSL_2(\mathbb F_{37})\\
43 & \PGL_2(\mathbb F_{43}) & \PGL_2(\mathbb F_{43}).
\end{array}
\]
In particular, modulo $13$, the two normalized representations have
different finite images:
\[
        \overline{\Gamma}_{54}^{(13)}
        =
        \PGL_2(\mathbb F_{13}),
        \qquad
        \overline{\Gamma}_{18}^{(13)}
        =
        \PSL_2(\mathbb F_{13}).
\]
This is the finite-image computation used in
Section~\ref{sec:distinguishing-fermat-orbits}.

\section{Fermat determinant computation}
\label{app:fermat-determinants}

The script \texttt{fermat\_determinants.py} takes as input the normalized matrices computed in Appendix~\ref{app:fermat-normalization} and computes the determinants of the ten pairwise differences among $0,\  I_2,\  M_2,\  M_3,\  M_4.$
These are the determinant data used in Section~\ref{sec:determinant-square-classes}.
The computation is carried out in the Eisenstein integer ring  $\mathbb Z[\omega],\  \omega^2+\omega+1=0,$
with elements represented as pairs $a+b\omega$.
\subsection*{Output}
Running the script gives the following output:

\begin{verbatim}
orbit 54 determinant list:
[1, 1, -1, -1, 3, 1+2*omega, -1-2*omega, -1-2*omega, 1+2*omega, -3]
orbit 54 determinant set:
[-3, -1-2*omega, -1, 1, 1+2*omega, 3]

orbit 18 determinant list:
[1, 1, 1, 1, 1, 1, 1, 3, 3, 3]
orbit 18 determinant set:
[1, 3]
\end{verbatim}
These are the determinant data used in Section~\ref{sec:determinant-square-classes}.

\section{Modulo $13$ normalization check}
\label{app:fermat-mod13}

The determinant computation in Appendix~\ref{app:fermat-determinants} uses one chosen normalization for each of the two Fermat orbit representatives. The script \texttt{fermat\_mod13\_normalizations.py} performs an additional finite-field check modulo $13$. For each representative sixer, it runs through all ordered choices of the normalization triple  $(L_0,L_\infty,L_1).$
There are $6\cdot 5\cdot 4=120$ such choices for each sixer. We reduce by sending
        $\omega\longmapsto 3\in\mathbb F_{13},$
since $ 3^2+3+1\equiv 0\pmod{13}.$
For each normalization, the script computes the determinants of the ten
pairwise differences among
        $0,\  I_2,\ M_2,\  M_3,\  M_4$
and records how many of these determinants are nonsquares in $\mathbb F_{13}$.

\subsection*{Output}

Running the script gives the following output:

\begin{verbatim}
orbit 54 nonsquare determinant counts: {4: 72, 6: 48}
orbit 18 nonsquare determinant counts: {0: 120}
\end{verbatim}
Thus, for the orbit of size $54$, $72$ of the $120$ normalizations
yield four nonsquare determinants, while the remaining $48$ yield six.
For the orbit of size $18$, all $120$ normalizations yield only square
determinants.
These are the normalization profiles used in
Section~\ref{sec:distinguishing-fermat-orbits} for the modulo $13$ check.

\section{Singular computations for the plane blow-up models}
\label{app:singular-blowup}

The computations described in Section \ref{sec:blowup-interpretation} were carried out and tested in \textsc{Singular}, version~4.4.1. Where necessary, the scripts include separate commands intended to accommodate the output format used by earlier versions of \textsc{Singular}. The complete source files and the corresponding output are available in the accompanying repository \cite{FermatSixersCode}. In this appendix we describe the mathematical structure of the computations.

\subsection*{Construction of the two blow-up models}

The first script starts with the configuration
\[Z_{18}=\left\{
\begin{aligned}
&(1:0:0),\ (0:1:0),\ (0:0:1),\ (1:1:1),\\
&(1:\omega:\omega^2),\ (1:\omega^2:\omega)
\end{aligned}
\right\}.
\]
It computes a basis of the cubic forms vanishing at these six points and constructs the graph of the associated anticanonical map
\[\Bl_{Z_{18}}\PP^2\longrightarrow\PP^3.\]
The image is the cubic surface \(y^2z-yz^2-x^2w+xw^2=0.\) 

The script then constructs the $27$ lines on this surface in the standard
blow-up notation
\[E_i,\qquad L_{ij},\qquad Q_i,\]
and identifies them with the explicitly indexed Fermat lines. The intersection relations among these curves are used to generate all $72$ sixers. For the sixer \(\mathcal S_{54}=\{L_{12},L_{16},L_{26},Q_3,Q_4,Q_5\},\) the script constructs the linear system \(\left |4H-2E_1-2E_2-E_3-E_4-E_5-2E_6 \right|.\)
In the original plane this is the system of quartics having double points at $p_1,p_2,p_6$ and passing through $p_3,p_4,p_5$. The images of the six curves in $\mathcal S_{54}$ under the resulting contraction give the configuration $Z_{54}$.
Blowing up $Z_{54}$ and applying its anticanonical map produces a second cubic surface in $\PP^3$. The script verifies its projective equivalence with the original model by an explicit linear transformation. It also checks that the six exceptional curves of the second blow-up correspond, under this transformation, to
\[
L_{12},\quad L_{16},\quad L_{26},\quad
Q_3,\quad Q_4,\quad Q_5.
\]

The auxiliary procedures in the script implement the standard operations needed in these computations: constructing lines and conics through selected points, computing graph ideals of rational maps, eliminating source coordinates, and comparing homogeneous ideals.

\subsection*{Projective automorphisms of the point configurations}

Two further scripts compute
\[\Aut_{\PP^2}(Z_{18}) \qquad \text{and}\qquad \Aut_{\PP^2}(Z_{54}).\]
For each of the $6!=720$ permutations $\sigma$ of the six points, the scripts introduce a generic $3\times 3$ matrix $T$ and six auxiliary scalars $t_1,\ldots,t_6$, and impose the equations
\[
T P_i=t_iP_{\sigma(i)}, \qquad i=1,\ldots,6.
\]
The parameters $t_i$ are eliminated, and saturation is used to remove the component corresponding to the zero matrix. A permutation is retained exactly when the resulting system admits a nonzero solution. One representative matrix is then obtained by choosing an affine chart in the projective space of matrices.

The computations give
\[|\Aut_{\PP^2}(Z_{18})|=36, \qquad |\Aut_{\PP^2}(Z_{54})|=12.
\]
Together with \(|\Aut(S_F)|=648,\) these values recover the orbit sizes
\[
\frac{648}{36}=18,\qquad\frac{648}{12}=54.
\]


\begin{thebibliography}{99}

\bibitem{politus3}
L.~Chiantini, {\L}.~Farnik, G.~Favacchio, B.~Harbourne,
J.~Migliore, T.~Szemberg, and J.~Szpond,
\emph{Combinatorics of skew lines in $\mathbb P^3$ with an application to algebraic geometry},
arXiv:2308.00761v2, 2025.

\bibitem{politus7}
L.~Chiantini, {\L}.~Farnik, G.~Favacchio, B.~Harbourne,
J.~Migliore, T.~Szemberg, and J.~Szpond,
\emph{Enumerative geometry of skew lines in $\mathbb P^3$ with a given associated finite group},
arXiv:2607.03539, 2026.


\bibitem{Singular}
W.~Decker, G.-M. Greuel, G.~Pfister, oraz H.~Sch\"onemann.
\newblock {\sc Singular} {4-4-1} --- {A} computer algebra system for polynomial computations, 2025.

\bibitem{Dolgachev}
I.~V. Dolgachev,
\emph{Classical Algebraic Geometry: A Modern View},
Cambridge University Press, 2012.

\bibitem{DolgachevDuncan}
I.~V. Dolgachev and A.~Duncan,
\emph{Automorphisms of cubic surfaces in positive characteristic},
Izvestiya: Mathematics \textbf{83} (2019), no.~3, 424--475.
Also available as arXiv:1712.01167.


\bibitem{Favacchio2025}
G.~Favacchio,
\emph{Finite subgroups of $\operatorname{PGL}_2(K)$ arising from configurations of skew lines in $\mathbb P^3_K$},
arXiv:2512.19811, 2025.


\bibitem{FermatSixersCode}
G.~Favacchio and G.~Malara,
\emph{Computational files for ``On the Geometry of Sixers on the Fermat Cubic Surface''},
GitHub repository,
\url{https://github.com/GrzMal/sixers-on-the-fermat-cubic-surface}.

\bibitem{HarrisAG}
J. Harris.
\newblock Algebraic Geometry: A First Course.
Graduate Texts in Mathematics, Vol.~133,
Springer-Verlag,
1992.

\bibitem{Hartshorne}
R.~Hartshorne,
\emph{Algebraic Geometry},
Graduate Texts in Mathematics, vol.~52,
Springer, New York, 1977.

\bibitem{HoltEickOBrien}
D.~F. Holt, B.~Eick, and E.~A. O'Brien,
\emph{Handbook of Computational Group Theory},
Discrete Mathematics and its Applications,
Chapman \& Hall/CRC, 2005.

\bibitem{Klein1873}
F.~Klein,
\emph{Ueber Fl\"achen dritter Ordnung},
Mathematische Annalen \textbf{6} (1873), 551--581.


\bibitem{MaclachlanReid}
C.~Maclachlan and A.~W. Reid,
\emph{The Arithmetic of Hyperbolic 3-Manifolds},
Graduate Texts in Mathematics, Vol.~219,
Springer, 2003.




\bibitem{Manin}
Yu.~I. Manin,
\emph{Cubic Forms: Algebra, Geometry, Arithmetic},
2nd ed., North-Holland Mathematical Library, Vol.~4,
North-Holland, Amsterdam, 1986.

\bibitem{Schlaefli}
L.~Schl\"afli,
\emph{An attempt to determine the twenty-seven lines upon a surface of the
third order, and to derive such surfaces in species, in reference to the
reality of the lines upon the surface},
Quarterly Journal of Pure and Applied Mathematics \textbf{2} (1858), 110--120.


\bibitem{Sims}
C.~C. Sims,
\emph{Computation with Finitely Presented Groups},
Encyclopedia of Mathematics and its Applications, Vol.~48,
Cambridge University Press, 1994.


\end{thebibliography}
\end{document}